\documentclass[10pt,letterpaper]{amsart}

\usepackage[final
]{showkeys}

\usepackage{AKstyle}

\numberwithin{equation}{section}

\usepackage{caption}
\usepackage[labelfont=rm]{subcaption}

\usepackage[pagebackref=true, colorlinks]{hyperref}

\hypersetup{pdffitwindow=true,linkcolor=blue,citecolor=blue,urlcolor=blue,menucolor=blue}

\usepackage{comment}

\newcommand{\mat}[1]{\left(\begin{matrix} #1 \end{matrix} \right)}  

\begin{document}

\author{Alex Kontorovich}
\thanks{Kontorovich is partially supported by NSF grant DMS-2302641, BSF grant 2020119, and a Simons Fellowship.}
\email{alex.kontorovich@rutgers.edu}
\address{Department of Mathematics, Rutgers University, New Brunswick, NJ}
\author{Xin Zhang}
\thanks{Zhang is partially supported by ECS grant 27307320 and NSFC grant 12001457.}
\email{xzhang@maths.hku.hk}
\address{Department of Mathematics, The University of Hong Kong, Pokfulam, Hong Kong}

\title[On the Local-Global Conjecture for  the Regular Pentagon]{On the Local-Global Conjecture for Combinatorial Period Lengths of Closed Billiards on the Regular Pentagon}

\begin{abstract}
We study the set of combinatorial lengths of asymmetric periodic trajectories on the regular pentagon, proving a density-one version of a conjecture of Davis-Lelievre.
\end{abstract}
\date{\today}
\maketitle

\section{Introduction}


By a polygonal billiard table, we mean a closed region in the plane bounded by a  polygon. Beginning with a point in the table and an initial direction, we imagine a straight line trajectory moving at unit speed continuing until reaching a wall of the polygon, whereupon the trajectory bounces off the wall with angle of incidence equal to angle of reflection and continues at unit speed. By convention, a trajectory stops if it strikes a vertex of the polygon. A closed billiard path is one that returns to its initial configuration in the tangent bundle after finite time. Its (geometric) period length is the Euclidean distance it travels until this first return. Its {\it combinatorial} period length is the number of walls of the polygon that are hit before first return.

 Polygonal billiards provide simple yet extremely rich examples of dynamical systems. When the polygonal billiards are rational, that is, all dihedral angles are rational multiples of $\pi$, one can say much more about the dynamical properties of these billiards, by studying their associated translation surfaces. For instance, if this surface is Veech, that is, if its symmetry group is a {\it lattice} in $\text{SL}_2(\mathbb R)$, then every trajectory is either closed or dense; moreover, the closed trajectories have quadratic asymptotic growth ordered by geometric period \cite{Veech1989}. Of particular interest are billiards on regular polygons \cite{Veech1992}; since the equilateral triangle and square tile the plane, the first  case of an ``exotic'' translation surface is the regular pentagon.   \par

The pentagonal billiard has been extensively studied in, e.g., \cite{DavisFuchsTabachnikov2011, DavisLelievre2018}. One question in particular raised in Davis-Lelievre's work is to determine the set of combinatorial period lengths of closed trajectories on the regular pentagon.
 
\begin{figure}[htpb]

	\includegraphics[width=0.3\textwidth]{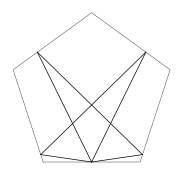}
	\includegraphics[width=0.3\textwidth]{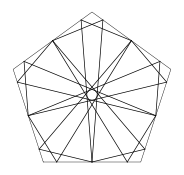}
	\includegraphics[width=0.3\textwidth]{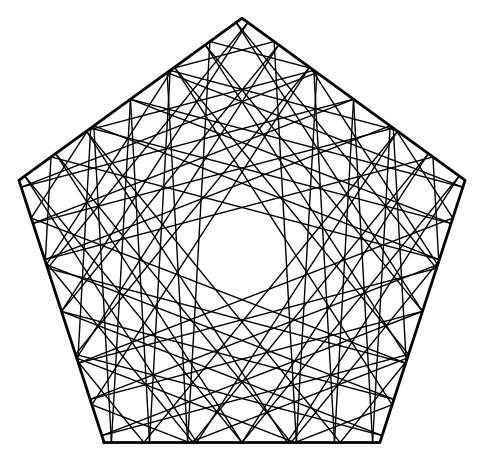}
	\caption{Periodic trajectories on a pentagonal billiard table with combinatorial period lengths 6, 20, 80, resp. Images are from \cite{DavisLelievre2018}.}
         \label{fig1}

\end{figure}

In \figref{fig1}, we show several periodic trajectories on the pentagon. The first is composed of 6 line segments, and hence has combinatorial period length 6. The second and third trajectories shown evidently have 5-fold rotational symmetry; we call such trajectories {\it symmetric}. The first image, by contrast, is  {\it asymmetric}. More details about such trajectories are reviewed in \secref{sec:billiards}.
After extensive numerical analysis, Davis-Lelievre were led to the following.

\begin{conj}[{\cite[Conjecture 4.3]{DavisLelievre2018}}]\label{conj:DL}
Every sufficiently large even number arises as an asymmetric combinatorial period length on the regular pentagon. More precisely, every positive even integer except  2, 12, 14 and 18 occurs.
\end{conj}
Our main result here is to establish  a density-one version of this conjecture.
\begin{thm}\label{thm:main1}
\conjref{conj:DL} is true in density, that is,
almost every even number arises. 

More precisely, the cardinality of the set of even numbers up to a parameter $X\to\infty$ which do not occur is bounded by $\ll X/(\log X)^{1/40}$.
\end{thm}

This problem can be naturally reformulated in the language of ``local-global phenomena on thin orbits,'' see \cite{Kontorovich2013} for more discussion. The underlying action is of a certain sub-semigroup, $\G^+$, of the $(2,5,\infty)$ Hecke triangle group, $\G$; see \secref{sec:billiards} below for details. The group $\G$ is a {\it lattice} (that is, acts discretely with finite covolume) for its action on $\SL_2(\R)$, but is ``thin''; that is, it has Zariski closure $\SL_2\times\SL_2$ (acting by the Galois conjugate in the second factor), and is an {\it infinite} index subgroup of the corresponding arithmetic group $\SL_2(\Z[\phi])$. Here $\phi=\frac{1+\sqrt5}2$ is the golden mean. For much more on the geometry and arithmetic of $\G$, and connections to Hodge theory, Teichm\"uller curves, and heights on Abelian varieties, see the recent works of McMullen \cite{McMullen2022, McMullen2023}.

What's most relevant for our purposes is the existence of a piecewise $\Q$-linear map 
$$
\sL:\G^+ \to \Z
$$ 
which takes an element $\g\in\G^+$ to the corresponding combinatorial period length $\sL(\g)$; then \conjref{conj:DL} asks to identify the image $\sL(\G^+)$ in $\Z$. There is an obvious ``local obstruction'' at the prime $2$, namely, that $\sL(\G^+)\subset 2\Z$. 
It is not hard to show that this is the only obstruction.
\begin{Def}
An integer $n$ is called {\it admissible} if $n\in\sL(\G^+)(\mod q)$, for all $q\ge1$.
\end{Def}
By an explicit analysis of the Strong Approximation property for $\G^+$ carried out in \secref{sec:strongApprox}, one can prove the following.
\begin{thm}\label{thm:Local}
An integer $n$ is admissible if and only if it is even.
\end{thm}

\begin{rmk}
One may think that, given this reformulation as a local-global problem in thin orbits, this problem is ripe for attack using the orbital circle method, as developed in, e.g., \cite{BourgainKontorovich2014a, BourgainKontorovich2014, FuchsStangeZhang2019}. 
Perhaps it could at least be used to improve \thmref{thm:main1} from a power of log savings on the exceptional set to a full power savings?
Indeed, the number of points in $\G^+$ in an archimedean ball grows quadratically (this is closely related to $\G$ being a lattice in $\SL(2,\R)$), and the map $\sL$ is linear, and hence one could expect that the {\it multiplicity} with which a large admissible integer occurs grows  linearly. (Compare this, say, with the local-global problem for Apollonian packings, where the average multiplicity of an admissible integer $n$ is only of order about $n^{0.3\dots}$; see, e.g., \cite{Kontorovich2013}.) With so many expected representations, not to mention the parabolic elements in $\G^+$, one may hope that the orbital circle method should perhaps be able to prove the full \conjref{conj:DL}.

But despite the many recent advances in this technology, we are sadly unable to apply it  in the present setting. 
One problem is that
 the semigroup $\G^+$ is just too thin:
 while the count for $\G^+$ in an archimedean ball grows quadratically, the ambient arithmetic group $\SL_2(\Z[\phi])$ has quartic growth. (See \cite{KontorovichZhang2024} where the authors are able to execute the orbital circle method with nearly square-root thinness in this sense.) 
Another major obstruction is that we still do not know the expander property for the family of  congruence quotients in $\G^+$  for {\it all} (not just square-free) moduli; this again is closely related to the Zariski closure being $\SL_2\times\SL_2$. 
\end{rmk}

Worse yet, it was recently discovered in \cite{HaagKertzerRickardsStange2023, RickardsStange2024} that the general local-global type statement in orbits may be false! It turns out that there can be further Brauer-Manin-type reciprocity obstructions in the orbit. Alas, our setting is no different; we show in \secref{sec:recip} the following.

\begin{thm}[{See \cite[Theorem 2.6]{RickardsStange2024}}]\label{thm:locGlobFail}
Let
\be\label{eq:gL}
\Lambda=\left\langle \mat{1&4\\0&1},  \mat{1&0\\4&1} \right\rangle^+
\ee
be generated as a semigroup in the given matrices, and let the linear form $\sL$ be given by:
$$
\sL\ : \ \gL\to \Z \ : \ \g\mapsto (1,0)\cdot \g\cdot (3,5)^t.
$$
Then the image
$$
\sL(\gL) \subset \Z
$$
has infinitely many local-global failures. In particular, $\sL(\gL)$
contains a density one subset of its admissibles, but it misses all the squares (which are not ruled out by congruence conditions). 
\end{thm}

In particular, unless the general orbital circle method  can somehow be made to distinguish between the settings in \conjref{conj:DL} and \thmref{thm:locGlobFail}, it will not succeed on this problem.
\\

Instead, our proof of \thmref{thm:main1} follows directly from an application of sieve methods, see \secref{sec:pfMain}. 
We show there that the set of asymmetric combinatorial periods contains the values of certain ternary cubic polynomials, which are linear in each variable.
(It is a rather interesting problem in its own right to understand how many admissible values may be missed from such a ternary cubic, and we discuss this question in \secref{sec:cubics}.)
Then some tools from  sieve theory show that such take on a density one set of admissible values, proving \thmref{thm:main1}.
Our crude method of producing periods is guaranteed to miss {\it at least} $\gg X/\log X$ even values up to $X$ (see \rmkref{rmk:missed}), and so other ideas are needed to go beyond a power of log savings on \conjref{conj:DL}.
Alas, as we show in \secref{sec:recip}, even this approach seems not to be able to avoid the possibility of reciprocity obstructions, see \corref{cor:cubicFail}.
\\

Some remarks:

\begin{rmk}
It would be interesting to investigate similar questions for combinatorial period lengths for other $n$-gons.
\end{rmk}

\begin{rmk}
In \cite{McMullen2023a},  McMullen shows that it is possible to filter periodic geodesics in an $n$-gon not only by whether they have rotational symmetry, but also by which edge midpoint gives a ``vertex connection,'' that is, a trajectory that begins and ends at a vertex. It would be interesting to investigate combinatorial period lengths filtered by this second invariant, which is also determined by a congruence condition. 
\end{rmk}

\subsection*{Acknowledgements} 
The authors are grateful to
Jayadev Athreya,
Henryk Iwaniec,
Curt McMullen,
Peter Sarnak,
and Kate Stange
for many comments and suggestions that improved this paper.
This paper was written while the first-named author was visiting Princeton
University; he would like to express his gratitude for their hospitality.
\\

\section{Combinatorial Periods on the Pentagon}\label{sec:billiards}

In this section, we record the relevant theorem 
on combinatorial periods
for our analysis.  To study periodic orbits in a regular pentagon, it is convenient to instead move to understanding saddle connections on an associated translation surface, the golden L, which is obtained by some transformations of the double pentagon. Here is the image in reverse:  \par

\begin{figure}[htbp]
	\centering
	\includegraphics[width=0.6\textwidth]{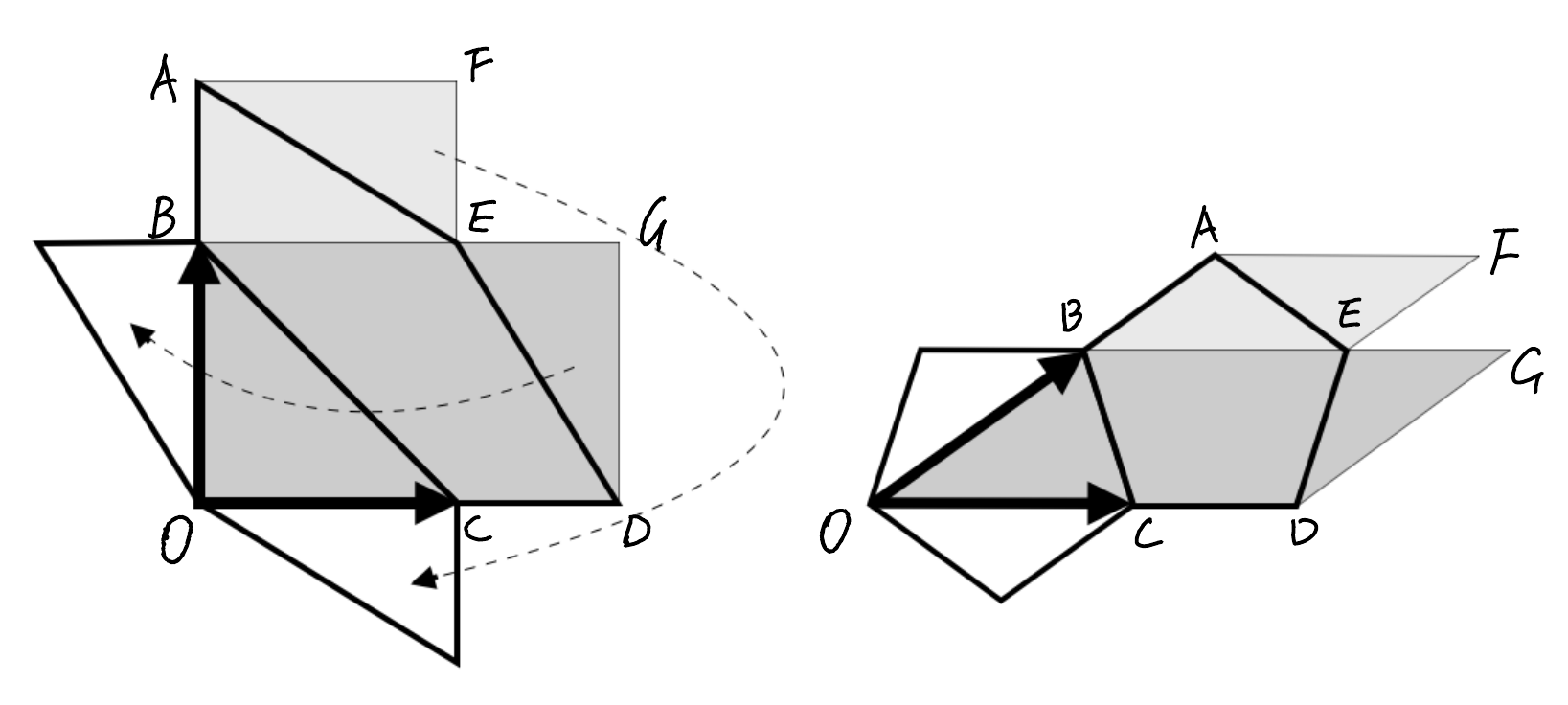}
	\caption{Transformations taking the golden L to the double pentagon. Image taken from \cite{DavisLelievre2018}.}
         \label{fig2}
\end{figure}

In particular, we cut and translate the triangles EGD and AFE from the golden L to form a sheared double pentagon, and then undo the shear by a linear transformation. Note that billiards on the pentagon correspond to geodesics on the double pentagon by the usual unfolding process, and geodesics are preserved under linear transformations and cut-and-translate operations.

The Veech symmetry group of golden L is $\Gamma:=\left\langle \mat{1&\phi\\0&1}, \mat{0&-1\\1&0} \right\rangle$. To parametrize periodic trajectories, it is more convenient to work instead with the semigroup: $$\Gamma^+=\left\langle \mat{1&\phi\\0&1}, \mat{1&0\\\phi&1}, \mat{\phi&\phi\\1&\phi},  \mat{\phi&1\\\phi&\phi} \right\rangle.$$ 
Each periodic direction in the golden L has a short and long saddle connection, and all long saddle connections with directions in the first quadrant are given by $\Lambda_L=\Gamma^+\cdot (0,1)^t$. The following theorem of Davis-Fuchs-Tabachnikov gives a relation between saddle connections in golden L and combinatorial period length in regular pentagon:

\begin{thm}[{\cite[Theorem 11]{DavisFuchsTabachnikov2011}}]\label{0606} Given a long saddle connection $(a+b\phi, c+d\phi)\in\Gamma^+\cdot (0,1)^t$ in the golden L, let $\ell=a+b+c+d$. The combinatorial period of the corresponding pentagon billiard trajectory is given either by $\sL=2\ell$ or $\sL=10\ell$, depending on whether $(d-b)+2(c-a)\equiv 0(\mod 5)$ or not (resp.), and corresponding to whether the associated trajectory is asymmetric or not.
\end{thm}

\proof[Sketch of proof]
A long saddle connection $(a+b\phi, c+d\phi)$ in the golden L starting from $O$ (as labeled in \figref{fig1}) must cross $ {EF}$ some number, $a$, times, $ {DG}$ $b$ times, $ {EG}$ $c$ times, and $ {AF}$ $d$ times, 
with $a$, $b$, $c$, and $d$ uniquely determined (due to the irrationality of $\phi$).
This implies the consecutive crossings of walls $ {BC}\rightarrow  {DE}$ $b$ times, $ {BC}\rightarrow  {AE}$ $d$ times, $ {CD}\rightarrow  {DE}$ $c$ times, and $ {AB}\rightarrow  {AE}$ $a$ times. So we have a total of $2(a+b+c+d)$ crossings.  On the regular pentagon table, the corresponding trajectory makes a counterclockwise angle $-\frac{2\pi}{5}$ when hitting $ {BC}$ followed by $ {AF}$, an angle $\frac{2\pi}{5}$ when hitting $ {BC}$ followed by $  {DE}$,  an angle $\frac{4\pi}{5}$ when hitting $ {CD}$ followed by $ {DE}$ and $-\frac{4\pi}{5}$ when hitting $ {AB}$ followed by $ {AE}$. The trajectory is already closed if and only if the total angle is an integer multiple of $2\pi$, that is, if $d-b+2(c-a)\equiv 0(5)$; these are the asymmetric trajectories. Otherwise, the trajectory's total angle is an integer multiple of $2\pi /5$ (but not of  $2\pi$), and hence the billiard has to continue for another four such trips in order to close up, increasing the length by a factor of $5$, and becoming symmetric. This explains \thmref{0606}.
\endproof

\ 

\section{Strong Approximation and Local Obstructions}\label{sec:strongApprox}

In this section, we work out the explicit local theory for the Hecke-5 semigroup $$\Gamma^+=\left\langle A=\mattwos1\phi01, B=\mattwos10\phi1, C=\mattwos\phi\phi1\phi, D=\mattwos\phi1\phi\phi \right\rangle$$ 
appearing in the previous section, and use this to prove  \thmref{thm:Local} on the local obstructions.\footnote{%
Strictly speaking, determining the local behavior of $\G^+$ is not necessary to identify the local obstructions for combinatorial periods, as it follows from our proof in \secref{sec:pfMain} of \thmref{thm:main1} that almost all even numbers are represented, and hence there can be no other local obstructions. But determining the reduction theory of $\G^+$ is of independent interest, so we include it here.}
Since we are interested in the local theory of $\Gamma^+$, it suffices to study the local theory of the group  $\Gamma = \<\Gamma^+\>$ generated by the semigroup, since they have the same modular reductions.
 For this purpose, it is convenient to view $\Gamma$ as a $\mathbb Q$-group.  Then the  $\Q$-points of its Zariski closure is $\text{SL}_2(\mathbb Q[\phi])$; let $G$ be the $\mathbb Z$ points of $\text{SL}_2(\mathbb Q[\phi])$. As usual, we write $\G(q)$ and $G(q)$ for the kernels of the mod $q$ reduction map, and $\G(\mod q)$ and $G(\mod q)$ for the images under this map.  The main result of this section is the following.
\begin{theorem}\label{1608} Let $q\in\mathbb Z$ have prime factorization $q=\prod p^e$. Then $\G(\mod q)$ is multiplicative:
$$
\G(\mod q)\cong \prod_p \G(\mod p^e).
$$
Moreover, set
 $q'=\normalfont\text{gcd}(q, 12)$.  Then 
$$\Gamma(q')/\Gamma(q)=G(q')/G(q).$$
So the local obstruction of $\Gamma$ (and hence $\Gamma^+$) is at $3$ and $2^2$.
Finally, we have that $\Gamma(\mod 12)$ is a subgroup of $G(\mod 12)$ of index $72$.
\end{theorem}
\thmref{1608} follows from a series of lemmas. 
\begin{lemma} \label{1618}
Let $q=2^s\prod_{i=1}^N p_i^{n_i}\in\mathbb Z_+$, $p_i$ odd, and $q_0=2^s\prod_{i=1}^N p_i$. Then 
$$\Gamma(q_0)/\Gamma(q)=G(q_0)/G(q).$$
That is, no new local obstructions occur in odd prime powers.
\end{lemma}

\begin{proof}
One basis for the Lie algebra $\mathfrak g$ of $G$ is given by 
$$\left\{a=\mattwos0100, b=\mattwos0\phi00, c=\mattwos0010, d=\mattwos00\phi0, e=\mattwos100{-1}, f=\mattwos\phi00{-\phi}\right\}$$
Write $\mathfrak g(\mathbb Z)=\mathbb Za\oplus \mathbb Zb\oplus \mathbb Zc\oplus \mathbb Zd\oplus \mathbb Ze\oplus \mathbb Zf$. Let $\frak g_1$ be the Lie algebra generated by $\Gamma(\{b,d \})$, the action here being the adjoint action.
For any prime $p$, any $m\geq 1$, and any $g\in G(p^m)/G(p^{m+1})$, we can write $g=I+p^m v$, where $v\in \mathfrak g(\mod p)$.  \par

We can compute
\begin{align*}
& g=BbB^{-1}=\mattwos{-1-\phi} \phi{-1-2\phi}{1+\phi}, \\
& h=CbC^{-1}=\mattwos{-1-\phi}{ 1+2\phi}{-\phi}{1+\phi}, \\
& i=DbD^{-1}=\mattwos{-1-2\phi} {1+2\phi}{-1-2\phi}{1+2\phi}.
\end{align*}
Since $\mathfrak g_1$ is invariant under transpose, $g^t, h^t, i^t \in \frak g_1$ as well. \par
Next note that 
\begin{align}
\label{1454}j=g+h-i-b+d=\mat{-1&0\\0&1},\\
\label{1455}k=h-2b+d-j=\mat{-\phi&1\\0&\phi},\\
\label{1456}l=g-b+2d-j=\mat{-\phi& 0\\-1& \phi}.
\end{align}
Therefore, 
\begin{align}
&m=k-l^t= \mat{0&2\\0&0}\in \frak g_1, \\
&n=m-2k= \mat{2\phi& 0\\0&-2\phi}\in \frak g_1.
\end{align}

Then $b, d, m, m^t, n, j$ generate a $\mathbb Z$-module that contains $2\mathfrak g(\mathbb Z)$. 
Therefore Hensel's lifting argument applies for any odd prime, and there are no local obstructions mod odd prime powers that are not due to an obstruction mod the prime itself.
\end{proof}

The next lemma deals with the prime 2. 
\begin{lemma}\label{1500}
For any $s\geq 2$, we have 
$$\Gamma (4)/\Gamma(2^s)= G(4)/G(2^s).$$
That is, there are no new obstructions at $p=2$ higher than $p^2$.
\end{lemma}
\begin{proof} 
We have that
$$A^{2^m} B^{2^n} A^{-2^m}B^{-2^n}\equiv I + 2^{m+n}\mat{\phi^2&0\\0&-\phi^2}(\mod 2^{m+n+\min\{m,n\}})$$
For the special case $m=n=1$, we have
\begin{align}
\label{1450}ABA^{-1}B^{-1}\equiv I +  4\mat{\phi^2&0\\0&-\phi^2}(\mod 8).
\end{align}
Then using $\phi^2=\phi+1$, $A$, $B$, \eqref{1454}, \eqref{1455}, \eqref{1456}  together with \eqref{1450}, we have that
\begin{align}
\Gamma(4)/\Gamma(8)=G(4)/G(8).
\end{align}
Similarly, we obtain for any $s\geq 2$ that
\begin{align}
\Gamma(2^s)/\Gamma(2^{s+1})=G(2^s)/G(2^{s+1}),
\end{align}
from which the claim follows.
\end{proof}

Our next goal is to show the following.
\begin{lemma} \label{1619}
 For any square free $q$ with $(q, 6)=1$, we have $\Gamma(\mod q)=G(\mod q)$. 
\begin{proof} 
We first note that 
$$\mat{1&\phi\\0&1}\cdot \mat{1&0\\\phi&1}\cdot \mat{1&-\phi\\0&1}\cdot \mat{1&0\\-\phi&1}\cdot \mat{1&\frac{1}{4}\phi\\0&1}=\mat{4+4\phi&0\\1+2\phi&\frac{1}{2}-\frac{\phi}{4}}:=M$$
can be obtained by a product of the generators $A$ and $B$ (since $(q,2)=1$, the prime $2$ is invertible).
Therefore, we have 
$$M^{-1}\cdot \mat{1&0\\\phi&1}\cdot M=\mat{1&0\\48+80\phi&1}.$$
Using this element and $B$ shows that if $p\neq 2,3$ a prime, then 
$$\mat{1&0\\\mathbb Z+\mathbb Z\phi&1}(\mod p)\subset \Gamma(\mod p).$$
By closure under transpose, we have that
$$\mat{1&\mathbb Z+\mathbb Z\phi\\0&1}(\mod p)\subset \Gamma(\mod p).$$

From the computation 
$$\mat{1&x\\0&1}\cdot \mat{1&0\\y&1}\cdot \mat{1&z\\0&1}=\mat{1+xy& (1+xy)z+x\\y&1+yz},$$
we see that for any $\mat{a&b\\c&d}\in G(\mod p)$ with $c\neq 0$,  we have $$\mat{a&b\\c&d}\in G(\mod p) \in \Gamma(\mod p).$$  
Next we show that the group generated by $A$ and $B$ in $G(\mod p)$ has cardinality exceeding half the size of $G(\mod p)$, and hence is all of $G(\mod p)$.
Indeed, if $\mathbb Z[\phi]/(p)=P_1P_2$ for two different prime ideals of norm $p$, then $$|G(\mod p)|=p^2(p-1)^2(p+1)^2,$$ and the number of elements in $G(\mod p)$ with 2-1 entry invertible is $$p^4(p-1)^2> \frac{p^2(p-1)^2(p+1)^2}{2}$$ if $p\geq 5$. \par
 If instead $(p)$ remains prime, then $$|G(\mod p)|=p^2(p^2-1)(p^2+1),$$ and the number of elements in $G(\mod p)$ with 2-1 entry invertible is $$p^4(p^2-1)> \frac{p^2(p^2-1)(p^2+1)}{2}$$ if $p\geq 5$. \par
In either case, more than half of the group is generated this way; and hence $A$ and $B$ generate all of $G(\mod p).$
\end{proof}
\end{lemma}

\thmref{1608} now follows from \lemref{1618}, \lemref{1500}, \lemref{1619}, and explicit computations.

\proof[Proof of \thmref{thm:Local}]
By \thmref{0606}, we need to understand the reduction of $\cO=\Gamma^+\cdot(0,1)^t$ mod $q$, which follows from knowing $\Gamma^+(\mod q)=\G(\mod q)$. It suffices to check at the ``bad'' moduli $q=2^2, 3,$ and  $5$ (the last modulus $p=5$ because it determines the values of $\sL=2\ell$ or $\sL=10\ell$) that there are no obstructions, and indeed an explicit calculation shows that a combinatorial period must be even, and this is the only local obstruction.
\endproof

Added in proof:  Calegari \cite{Calegari2024} has determined the adelic closure of triangle groups corresponding to all $n$-gons, of which our analysis is a special case.

\

\section{Proof of \thmref{thm:main1}}\label{sec:pfMain}

To begin the proof, we note that the bottom row vector of 
\be\label{eq:trilinearProd}
B^m A^n B^k=
\mattwos10 {m \phi} 1\cdot \mattwos1{n\phi} 01\cdot \mattwos10{k\phi}1 \in \Gamma^+
\ee 
is 
$$
(mnk+ (2mnk+m+k)\phi, 
mn+1+mn\phi)=:(a+b\phi, c+d\phi).
$$

By \thmref{0606}, we need to consider the quantity
\be\label{eq:trilinear}
\ell=\ell(m,n,k)=a+b+c+d= 3mnk+2mn+m+k+1,
\ee
so that the corresponding period length is either $\sL=2\ell$ or $\sL=10\ell$, depending on whether
\begin{align}\label{1550}
d-b+2(c-a)\equiv 
g(m,n,k):=
mnk+ 3mn - m-k +2 
\equiv 0(\mod 5) 
\end{align}
or not; this corresponds to whether the associated billiard is asymmetric or not.
We can ensure that the period is $2\ell$ (and the associated billiard is asymmetric), by fixing $k$ and restricting $m, n$ to certain progressions, while ensuring that $\ell$ can take all possible values mod $5$, as follows:
\begin{enumerate}
\item If we take $m=5m'+3, n=5n'+1, k=1$, then $g(m,n,k)\equiv 0(\mod 5)$ and 
$$\ell=
5(25 m' n' + 6 m' + 13 n') + 20
\equiv0(\mod 5).$$
\item If we take $m=5m'+4$, $n=5n'+3, k=1$, then $g(m,n,k)\equiv 0(\mod 5)$ and 
$$
\ell=
5(25 m' n' + 16 m' + 20 n') + 66 \equiv 1 (\mod 5).
$$
\item If we take $m=5m'+1, n=5n'+2, k=3$, then $g(m,n,k)\equiv 0(\mod 5)$ and 
$$\ell=
5(55 m' n' + 23 m' + 11 n') + 27
\equiv 2(\mod 5).
$$
\item If we take $m=5m'+1, n=5n'+3, k=0$, then $g(m,n,k)\equiv 0(\mod 5)$ and 
$$\ell =
5(10 m' n'+
 7 m' + 2 n') + 8
\equiv 3 (\mod 5).$$
\item If we take $m=5m'+2, n=5n'+2, k=1$, then $g(m,n,k)\equiv 0(\mod 5)$ and 
$$\ell=
5(25 m' n'+
 11 m' + 10 n') + 24
\equiv 4 (\mod 5).$$
\end{enumerate}
These conditions ensure that $\ell$ can cover all five residue classes mod $5$.
\thmref{thm:main1} will easily follow from the following.

\begin{lem} \label{0827}Let $l(x, y)= Axy+Bx+Cy$, with $A, B, C\in \mathbb Z_+,$ and $\gcd(A, B)=1$. Then $l$ represents a density one subset of positive integers; in fact, the proportion of $n\le X$ represented by $l$ is given by:
$$
\frac1X\#\{n \le X : \exists x,y,\text{ such that } n=l(x,y)\} \  = \  1 + O(\log^{-1/\varphi(A)}X),
$$
as $X\to\infty$. Here $\varphi$ is the Euler totient function.
\end{lem}
\begin{proof} Write 
$l(x, y)=\frac{(Ay+B)(Ax+C)-BC}{A}.$
It suffices to show that almost every $n\equiv BC(\mod A)$ can be written as $(Ay+B)(Ax+C)$.  
Thus it would be sufficient to show that almost every $n\equiv BC(\mod A)$ contains a prime factor $p\equiv B(\mod A)$; indeed, this would then determine the value of $y$, and $n/p\equiv C(\mod A)$ (since $A$ and $B$ are coprime) determines the value of $x$. We claim that the proportion of $n$ up to a growing parameter $X$ which have {\it no} prime factor $p\equiv B(\mod A)$ has cardinality $\ll X(\log X)^{-1/\varphi(A)}$.

To see this, let 
$$
f(n)=\bo_{\{\forall p\mid n, p\not\equiv B(\mod A)\}}$$ 
be the (multiplicative) indicator function in question,
and apply standard results in sieve 
bounds, for example \cite[Equation (1.85)]{IwaniecKowalski}, to see that
$$
\#\{n\le X : \forall p\mid n, p\not\equiv B(\mod A)\} =\sum_{n\le X}f(n)
\ll {X\over \log X} \prod_{p<X}\left(1+\frac{f(p)}p\right)
 \asymp X \prod_{p<X \atop p\equiv B(\mod A)}\left(1-\frac1p\right).
$$
As is classical, for any $B$ coprime to $A$, the product above is of order $(\log X)^{-1/\varphi(A)}$, as $X\to\infty$.
This completes the proof.
\end{proof}

\proof[Proof of \thmref{thm:main1}]
We apply \lemref{0827} to the various expressions for $\ell$ above. For example, when $\ell = 5(25m'n'+6m'+13n')+20 = 5l+20$, the lemma shows that $l$ represents all but $X(\log X)^{-1/20}$ numbers up to $X$, since 
$\phi(25)=20$. Therefore $\ell$ represents all but this many numbers $\equiv0(\mod 5)$ up to $X$. Our various expressions for $\ell$ cover each of the five residue classes mod $5$ while keeping $g(m,n,k)\equiv0(\mod 5)$, and hence staying within asymmetric trajectories. Since $\phi(55)=40$ and $\phi (10)=4$, we see that almost every even number (that is, value of $\sL=2\ell$) occurs as a combinatorial period, with a exceptional set of cardinality at most $\ll X(\log X)^{-1/40}$, as claimed.
\endproof

\begin{remark}\label{rmk:missed}
Note that our proof of \lemref{0827} is guaranteed to miss at least $\gg X/\log X$ possible values of $l$, simply because all primes $n\equiv BC(\mod A)$ are automatically excluded.
So this method of producing combinatorial billiards is guaranteed to miss at least this many values. In the next section, we discuss how one can try to do better.
\end{remark}

\section{Ternary Cubics Linear in Each Variable}\label{sec:cubics}

It is a question of independent interest  to investigate the values  on the {\it positive} integers taken by integral ternary cubics that are linear in each factor, such as that in \eqref{eq:trilinear}. 
Note that in our analysis in \secref{sec:pfMain}, we fixed one of the variables, and were left with a quadratic polynomial as in \lemref{0827}. One may wonder whether one can do much better by exploiting all three variables.\footnote{Of course one can also try to use more variables in \eqref{eq:trilinearProd}, creating quaternary quartics, say, linear in each variable, whose values are combinatorial period lengths. We have not been able to make much use of these higher degree forms, and it may anyway in general not be possible; see \rmkref{rmk:higher}.}

Let
$F$ be an integral ternary cubic, linear in each factor, with nonnegative coefficients,
$$F(x,y,z)=axyz+bxy+cyz+dxz+ex+fy+gz+h,$$
with $a,b,...,h\ge0$. We want to study its values on positive entries, $x,y,z\ge 1$. Thus for every given value of $z=1,2,3,…$, $F_z$ is a reducible (over $\Q$) quadratic:
$$
F_z(x,y) = Axy+Bx+Cy+D = ((Ax+C)(Ay+B)+AD-BC)/A.
$$

As $x,y$ range in the positive integers, the set of $n$ represented by $F_z$ are of the form
$$nA-(AD-BC)=(Ax+B)(Ay+C).$$
So $n$ (up to linear transformation) is a product of linear terms. This restricts the possible values of $n$.

The question is already very interesting in the case that
$$
F(x,y,z)=xyz+x+y+z.
$$ 
\begin{problem}
Determine the numbers positively represented by $F(x,y,z)=xyz+x+y+z.$
\end{problem}
Note that if we allow nonnegative values of $x,y,z$, then $F(0,0,z)=z$, so every number is already represented! So the question is really only interesting for {\it positive} values of the variables. 

Let 
$$
\cN:=F(\Z_{\ge1},\Z_{\ge1},\Z_{\ge1})
$$ be the values positively represented by $F$, and let 
$$
\cN^c(X) := \#[1,X]\setminus\cN
$$ 
be the number of values up to $X$ not represented by $F$.
Also let $\cN_z(X)$ be the number $n<X$ not represented by $F_j$, for all $j\le z$.

For $z=1$, we have that
$$
F_1(x,y)=xy+x+y+1.
$$
So if $n=F_1$, then $n=(x+1)(y+1)$. Therefore all composite numbers are represented already by $F_1$. But all primes are missed. So the set of values {\it not} represented by $F$ is already bounded by 
$$
\cN^c(X)
\ll X/\log X,
$$
but
$$
\cN_1^c(X) \gg X/\log X.
$$

Now suppose that $n$ is prime; we look at $z=2$. Then
$$
F_2(x,y)=2xy+x+y+2 = ((2x+1)(2y+1)+3)/2.
$$
Since all composite numbers are already represented by $F_1$, we only need to look at which primes are represented by $F_2$.
These are the primes ($n=p$) having representations as:
$$
2p-3=(2x+1)(2y+1).
$$
Any $p\ne 3$ for which $2p-3$ is composite is thus represented! So the only numbers $n$ {\it not} represented by either $F_1$ nor $F_2$ are those $n=p$ for which $2p-3$ is also prime. (These are of course Sophie Germain-type primes.) A classical upper bound sieve then shows that 
$$
\cN^c(X)\ll X/(\log X)^2,
$$
and twin-prime-type  conjectures (such as the Hardy-Littlewood tuples conjecture) suggest that
$$
\cN_2^c(X) \gg X/(\log X)^2.
$$

Next we look at $z=3$. Then
$$
F_3(x,y)=3xy+x+y+3 = ((3x+1)(3y+1) + 8)/3.
$$
The numbers represented by neither $F_1$, $F_2$, nor $F_3$   are those $n=p$ for which $2p-3$ is also prime, and for which
$$
3p-8=(3x+1)(3y+1)
$$
has no representation.
To be represented by $(3x+1)(3y+1)$, a number that is itself $\equiv1(\mod3)$ must either have a nontrivial prime factor $\equiv1(3)$ or, if not that, then at least three (and hence at least four) prime divisors that are $\equiv2(\mod3)$. 
Thus the numbers $3p-8$ not represented are either primes $\equiv1(\mod3)$ or products of two primes $\equiv2(\mod 3)$. 
The cardinality of $n<X$ such that $n=p$ is prime, $2p-3$ is prime, and $3p-8$ is a product of at most two primes can again be upper-bounded, leading to
$$
\cN^c(X) \ll {X \log\log X\over (\log X)^3}.
$$

Getting a lower bound on $\cN_3^c(X)$ is more subtle.
The Hardy-Littlewood tuples conjecture says that the cardinality of $n<X$ such that $n=p$ is prime, $2p-3$ is prime, and $3p-8$ is prime is  $\gg X/(\log X)^3$. 
The other case is that $n<X$ is prime, $n=p$, $2p-3$ is also prime, and $3p-8$ is the product of exactly two primes, both $\equiv 2(\mod 3)$.
Recall that the number of integers up to $X$ which are the product of exactly two primes is asymptotic to $X\log\log X/\log X$.
 Consider for an integer $n$ the three events: (1) $n$ is prime, (2) $2n-3$ is prime, and (3) $3n-8$ is a product of exactly two primes (both of which should be $\equiv2(\mod 3)$ about half the time). Assuming these events are roughly independent, the cardinality of such $n<X$ may be expected to be: 
 $$
 \asymp X\cdot \frac1{\log X} \cdot \frac1{\log X}\cdot \frac{\log\log X}{\log X}=X\log\log X/(\log X)^3.
 $$
 Therefore we might expect in total that 
 $$
 \cN_3^c(X)\overset{?}{\gg}{X\log\log X\over (\log X)^3}.
 $$

One may wonder whether it might be possible to continue an analysis of this form and prove a full local-global theorem for such ternary cubics. We show in the next section that in fact this is, in general, impossible, and there may be further (reciprocity) obstructions.

\section{Reciprocity Obstructions}\label{sec:recip}

In this section, we give a proof of Theorem \ref{thm:locGlobFail}; note that it  is also a consequence of \cite[Theorem 2.6]{RickardsStange2024}, with a similar proof.
We begin with the following simple observation.

\begin{lemma}\label{lem:gL}
Let 
\be\label{eq:gL}
\Lambda=\left\langle \mat{1&4\\0&1},  \mat{1&0\\4&1} \right\rangle^+
\ee
be generated as a semigroup in the given matrices. Then every vector $(a,b)^t$ in the orbit $\cO$ given by:
$$
\cO:= \Lambda\cdot (3, 5)^t
$$ 
satisfies $ \left( \frac{a}{b}\right)=-1$, where $ \left( \frac{\cdot}{\cdot}\right)$ is the Jacobi symbol. In particular, the entry $a$ can never be a perfect square.
\end{lemma}
\begin{proof}
Indeed,
by quadratic reciprocity (for $q_1, q_2$ odd),
 \begin{align}\label{1658}
\left( \frac{q_1}{q_2}\right)\left( \frac{q_2}{q_1}\right)=(-1)^{\frac{q_1-1}{2}\frac{q_2-1}{2}},
\end{align}
 the Jacobi symbol is preserved under the transformations $(q_1, q_2)\mapsto (q_1+4q_2, q_2)$ and $(q_1, q_2)\mapsto (q_1, 4q_1+q_2)$. (This is a key tool underlying the work in \cite{RickardsStange2024}.) But the base vector  $(3,5)^t$ satisfies
$$ 
\left( \frac{3}{5}\right)=-1.
$$ 
Therefore in fact all vectors in $\cO=\Lambda \cdot (3,5)^t$ have Jacobi symbol $-1$, and so the entries can never be squares. 
\end{proof}

This will allows us to easily prove \thmref{thm:locGlobFail}.
\proof[Proof of \thmref{thm:locGlobFail}]
Let $\gL$ be as in \eqref{eq:gL} and the linear form $\sL$ given by:
$$
\sL:\gL\to \Z:\g\mapsto (1,0)\cdot \g\cdot (3,5)^t.
$$
Denote the image
$$
S := \sL(\G) \subset \Z.
$$

Consider the quadratic polynomial
$$
Q(x,y)=(1,0)\cdot \mat{1&0\\4x&1} \cdot \mat{1&4y\\0&1} \cdot (3,5)^t=4(20xy+3x)+5.
$$
Applying \lemref{0827} (using $\gcd(20,3)=1$), we see that $Q$ already represents a density one set of $n\equiv1(\mod 4)$, which is the set of admissibles for $S$.
However, by \lemref{lem:gL}, no squares appear in $S$.
 \endproof

In particular, this means that the analysis discussed in \secref{sec:cubics} on ternary cubics also cannot, in general, succeed to prove a full local-global principle.

 \begin{cor}\label{cor:cubicFail}
 Let
 $F$ be the following ternary cubic, linear in each variable, and with non-negative coefficients:
  $$F(x,y,z)= (0,1)\cdot \mat{1&0\\4x&1}\cdot \mat{1&4y\\0&1}\cdot \mat{1&0\\4z&1}\cdot (3,5)^t=64xyz+16xy+4x+4z+1.$$
 Then $S:=\{F(x,y,z):x,y,z> 0\}$ has admissible set $\{n\in\mathbb Z: n\equiv 1(4)\}$, but $S$ misses all squares. 
 \end{cor}
 
 \begin{rmk}\label{rmk:higher}
Of course, this argument can be repeated for $n$-ary forms of degree $n$, suggesting that such elementary methods will not resolve \conjref{conj:DL}.
\end{rmk}

\begin{rmk}
Closely related issues of local-global failures in ternary cubic forms arise in \cite{GhoshSarnak2022}; see also \cite{Harpaz2017} for other potential ``beyond Brauer-Manin'' type obstructions.
\end{rmk}

\bibliographystyle{alpha}

\bibliography{AKbibliog}

\end{document}